\documentclass[12pt]{amsart}
\usepackage{amsmath,amsfonts,amssymb,amsthm,amstext,pgf,graphicx,verbatim,lmodern,textcomp,color,young,tikz}
\usetikzlibrary{decorations}
\usepackage[mathscr]{euscript}
\usetikzlibrary{decorations.markings}
\usetikzlibrary{arrows}
\usepackage{float}
\usepackage{enumerate}
\usetikzlibrary{matrix, arrows.meta}
\usepackage{hyperref}
\theoremstyle{plain}
\newtheorem{theorem}{Theorem}[section]
\newtheorem{lemma}[theorem]{Lemma}

\theoremstyle{remark}

\theoremstyle{definition}

\newtheorem{example}{Example}[section]

\title{}
\begin{document}
	\title [Unimodular matrices and lattice paths enumeration via Pascal's triangle]{Unimodular matrices and lattice paths enumeration via Pascal's triangle}
	\author[S. Bera]{S. Bera}
	\address[Sudip Bera]{Faculty of Mathematics, Dhirubhai Ambani University, Gandhinagar, India}
	\email{sudipbera517@gmail.com, sudip\_bera@dau.ac.in}
	
	
	\begin{abstract}
		This article investigates a remarkable combinatorial identity involving a distinguished family of matrices whose entries are defined via binomial coefficients. Specifically, we consider a class of \( n \times n \) matrices parameterized by a positive integer \( m \), where each entry reflects a structured pattern derived from Pascal's triangle, particularly the diagonals corresponding to figurate numbers such as triangular, tetrahedral, and higher-dimensional simplex numbers. We establish, by means of a bijective argument, that the determinant of any such matrix is identically equal to \( 1 \), independent of the specific values of \( m \) and \( n \), provided that \( 2 \leq m \leq n \). This result unveils a profound connection between classical binomial identities and the enumeration of lattice paths in grid graphs.

		\medskip
		
		
		\medskip  
		
		
	\end{abstract}
	
	\maketitle

	\section{Introduction}
	A \emph{binomial determinant} usually refers to the determinant of a matrix where the entries are binomial coefficients \({n \choose k}\), or polynomials involving them. One of the most elegant representations of binomial coefficients is found in Pascal's triangle, where each entry $\binom{n}{k}$ corresponds to the number of paths in a grid or the number of ways to choose $k$ items from a set of $n$ elements.
	Pascal's triangle not only encodes combinations but also contains within its diagonals a series of well-known figurate numbers, such as the natural numbers, triangular numbers, tetrahedral numbers, and their higher-dimensional analogs. These diagonals can be interpreted geometrically in terms of the number of $k$-dimensional simplices that can be formed from a set of points. 
	
	Binomial determinants arise in the enumeration of
	\emph{plane partitions} fitting in an $a \times b \times c$ box, rhombus
	tilings of hexagonal regions, and lozenge tilings
	\cite{gessel1985,krattenthaler1999,krattenthaler2005,macmahon1897}.
	The classical MacMahon formula states that the number of plane partitions fitting
	in an $a \times b \times c$ box is
	\[
	\prod_{i=1}^{a}\prod_{j=1}^{b}\prod_{k=1}^{c}
	\frac{i+j+k-1}{i+j+k-2}.
	\]
	More generally, the number of plane partitions with a prescribed symmetry group
	can be expressed as a determinant whose entries are binomial coefficients
	\cite{krattenthaler1999,stembridge1990}.
	
	A matrix with integer entries and determinant equal to $\pm 1$ is called
	\emph{unimodular}, and is \emph{totally unimodular} (TU) if every square submatrix has
	determinant $0$, $+1$, or $-1.$ Totally unimodular matrices play a fundamental role in polyhedral combinatorics and combinatorial optimization, as they provide a convenient criterion for determining whether a linear program is integral \cite{HellerTompkins1956,HoffmanKruskal2010,HoffmanKruskal1956,Zaslavsky1982}. More precisely, if \(A\) is a totally unimodular matrix and \(b\) is an integral vector, then the linear programs $\min \left\{ c^{\top}x \;\middle|\; Ax \geq b,\; x \geq 0 \right\}$ and
	$\max \left\{ c^{\top}x \;\middle|\; Ax \leq b \right\}$
	have integral optimal solutions whenever an optimum exists, for every choice of the cost vector \(c\).
	Consequently, if \(A\) is totally unimodular and \(b\) is integral, then every extreme point of the feasible region
	$\left\{ x \;\middle|\; Ax \geq b \right\}$
	is integral. Therefore, the feasible region is an integral polyhedron.  
	Lattice path counting and binomial coefficients are central to enumerative combinatorics \cite{gessel1985,stembridge1990}. One of the deepest combinatorial interpretations of binomial determinants comes
	from the \emph{Lindström-Gessel-Viennot (LGV) lemma} (Lemma \ref{lgv-lemma}). 
	Applying LGV lemma to the integer lattice $\mathbb{Z}^2$ with unit weights,
	the entry $\binom{m+n-i-1}{m-1}$ counts the number of lattice paths (using unit
	steps East or North) from one fixed source to a target depending on $i$
	\cite{gessel1985,krattenthaler1999}. 
	
	In this article, we study a class of matrices constructed using entries derived from Pascal's triangle and analyze their determinants. Specifically, for a given integer $m \geq 2$, we define an $n \times n$ matrix $M = (m_{i,j})$ where each entry is based on a modified binomial expression. Surprisingly, despite the complex structure of the matrix, its determinant always equals 1.
	
	We begin by visualizing the diagonals of Pascal's triangle to highlight their relationship to classical figurate numbers. We then define the matrix family in detail, prove the main determinant identity, and explore specific examples, such as triangular and pentatope numbers appearing as matrix elements. These insights underscore the interplay between combinatorics, matrix theory, and geometric number sequences.
	
	The identity $\det(M) = 1$ in our main theorem (Theorem \ref{Thm:main-thm}) then asserts
	that there is essentially a unique family of non-intersecting lattice paths
	compatible with the endpoint constraints imposed by $M$, a remarkable
	combinatorial rigidity. 

	\begin{figure}[H]
		\begin{center}
			\begin{tikzpicture}[scale=0.65]
				\definecolor{ones}{RGB}{128, 128, 128}     
				\definecolor{natural}{RGB}{255, 0, 0}      
				\definecolor{triangular}{RGB}{0, 150, 0}   
				\definecolor{tetrahedral}{RGB}{0, 0, 255}  
				\definecolor{pentatope}{RGB}{148, 0, 211}  
				\definecolor{fivesimplex}{RGB}{255, 165, 0} 
				\definecolor{sixsimplex}{RGB}{255, 20, 147} 
				
				\node[color=ones] at (0,10) {1};
				
				\node[color=ones] at (-0.6,9) {1};
				\node[color=natural] at (0.6,9) {1};
				
				\node[color=ones] at (-1.2,8) {1};
				\node[color=natural] at (0,8) {2};
				\node at (1.2,8) {1};
				
				\node[color=ones] at (-1.8,7) {1};
				\node[color=natural] at (-0.6,7) {3};
				\node[color=triangular] at (0.6,7) {3};
				\node at (1.8,7) {1};
				
				\node[color=ones] at (-2.4,6) {1};
				\node[color=natural] at (-1.2,6) {4};
				\node[color=triangular] at (0,6) {6};
				\node at (1.2,6) {4};
				\node at (2.4,6) {1};
				
				\node[color=ones] at (-3,5) {1};
				\node[color=natural] at (-1.8,5) {5};
				\node[color=triangular] at (-0.6,5) {10};
				\node[color=tetrahedral] at (0.6,5) {10};
				\node at (1.8,5) {5};
				\node at (3,5) {1};
				
				\node[color=ones] at (-3.6,4) {1};
				\node[color=natural] at (-2.4,4) {6};
				\node[color=triangular] at (-1.2,4) {15};
				\node[color=tetrahedral] at (0,4) {20};
				\node at (1.2,4) {15};
				\node at (2.4,4) {6};
				\node at (3.6,4) {1};
				
				\node[color=ones] at (-4.2,3) {1};
				\node[color=natural] at (-3,3) {7};
				\node[color=triangular] at (-1.8,3) {21};
				\node[color=tetrahedral] at (-0.6,3) {35};
				\node[color=pentatope] at (0.6,3) {35};
				\node at (1.8,3) {21};
				\node at (3,3) {7};
				\node at (4.2,3) {1};
				
				\node[color=ones] at (-4.8,2) {1};
				\node[color=natural] at (-3.6,2) {8};
				\node[color=triangular] at (-2.4,2) {28};
				\node[color=tetrahedral] at (-1.2,2) {56};
				\node[color=pentatope] at (0,2) {70};
				\node at (1.2,2) {56};
				\node at (2.4,2) {28};
				\node at (3.6,2) {8};
				\node at (4.8,2) {1};
				
				\node[color=ones] at (-5.4,1) {1};
				\node[color=natural] at (-4.2,1) {9};
				\node[color=triangular] at (-3,1) {36};
				\node[color=tetrahedral] at (-1.8,1) {84};
				\node[color=pentatope] at (-0.6,1) {126};
				\node[color=fivesimplex] at (0.6,1) {126};
				\node at (1.8,1) {84};
				\node at (3,1) {36};
				\node at (4.2,1) {9};
				\node at (5.4,1) {1};
				
				\node[color=ones] at (-6,0) {1};
				\node[color=natural] at (-4.8,0) {10};
				\node[color=triangular] at (-3.6,0) {45};
				\node[color=tetrahedral] at (-2.4,0) {120};
				\node[color=pentatope] at (-1.2,0) {210};
				\node[color=fivesimplex] at (0,0) {252};
				\node at (1.2,0) {210};
				\node at (2.4,0) {120};
				\node at (3.6,0) {45};
				\node at (4.8,0) {10};
				\node at (6,0) {1};
				
				\node[color=ones] at (-6.6,-1) {1};
				\node[color=natural] at (-5.4,-1) {11};
				\node[color=triangular] at (-4.2,-1) {55};
				\node[color=tetrahedral] at (-3,-1) {165};
				\node[color=pentatope] at (-1.8,-1) {330};
				\node[color=fivesimplex] at (-0.6,-1) {462};
				\node[color=sixsimplex] at (0.6,-1) {462};
				\node at (1.8,-1) {330};
				\node at (3,-1) {165};
				\node at (4.2,-1) {55};
				\node at (5.4,-1) {11};
				\node at (6.6,-1) {1};
				
				\draw[ones, thick, ->] (-0.4,9.2) -- (-6.4,-0.8);
				\node[color=ones] at (-7.5,4) {\textbf{1st diagonal: All 1s}};
				
				\draw[natural, thick, ->] (0.4,8.8) -- (-5.2,-0.8);
				\node[color=natural] at (-7.5,7) {\textbf{2nd diagonal}};
				\node[color=natural] at (-7.5,6.5) {\textbf{Natural numbers}};
				
				\draw[triangular, thick, ->] (0.4,6.8) -- (-4.0,-0.8);
				\node[color=triangular] at (-8.5,8.5) {\textbf{3rd diagonal}};
				\node[color=triangular] at (-8.5,8) {\textbf{Triangular numbers}};
				
				\draw[tetrahedral, thick, ->] (0.4,4.8) -- (-2.8,-0.8);
				\node[color=tetrahedral] at (7.5,7) {\textbf{4th diagonal}};
				\node[color=tetrahedral] at (7.5,6.5) {\textbf{Tetrahedral numbers}};
				
				\draw[pentatope, thick, ->] (0.4,2.8) -- (-1.6,-0.8);
				\node[color=pentatope] at (8.5,5) {\textbf{5th diagonal}};
				\node[color=pentatope] at (8.5,4.5) {\textbf{Pentatope numbers}};
				
				\draw[fivesimplex, thick, ->] (0.4,0.8) -- (-0.4,-0.8);
				\node[color=fivesimplex] at (9,2.5) {\textbf{6th diagonal}};
				\node[color=fivesimplex] at (9,2) {\textbf{5-simplex numbers}};
				
				\draw[sixsimplex, thick, ->] (0.4,-1.2) -- (0.4,-1.2);
				\node[color=sixsimplex] at (8.5,0) {\textbf{7th diagonal}};
				\node[color=sixsimplex] at (9.5,-0.5) {\textbf{6-simplex numbers}};
				
			\end{tikzpicture}
			\caption{Pascal's Triangle with Seven Highlighted Diagonals}
			
		\end{center}
	\end{figure}
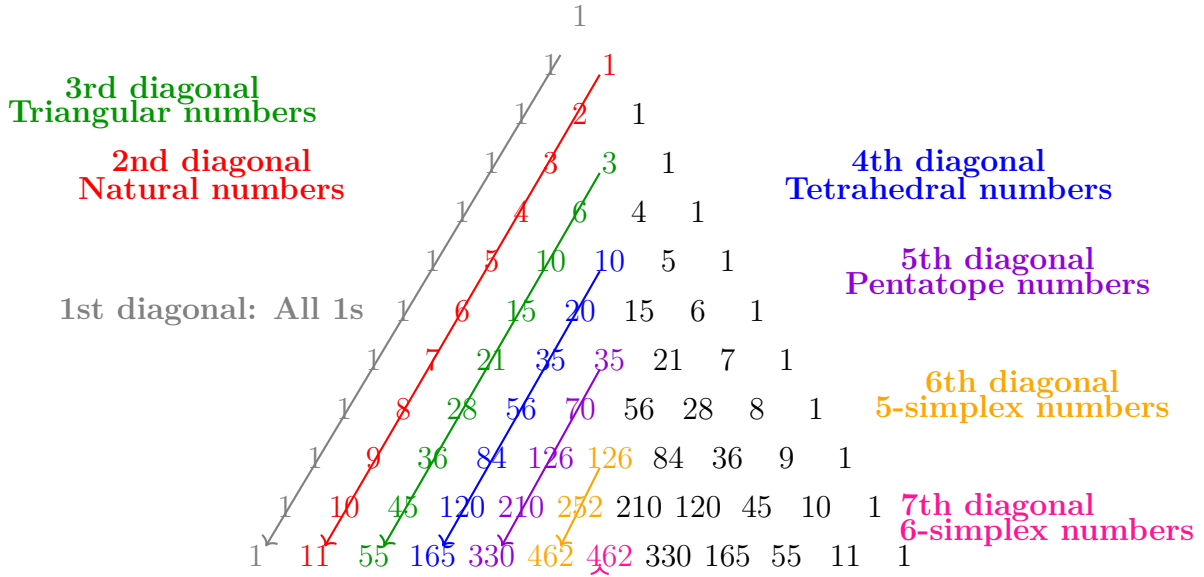
	
	\subsection{Formulas Along the Diagonals of Pascal's Triangle}
	
	Pascal's Triangle is not just a combinatorial tool-it's a rich tapestry of number patterns. Each diagonal corresponds to a familiar and meaningful sequence, many of which arise naturally in geometry and counting problems. Below, we explore the formulas and interpretations of the early diagonals, then generalize to higher-dimensional structures.
	
	\subsubsection*{1st Diagonal: All Ones}
	
	The first diagonal consists entirely of ones:
	\[
	\binom{n}{0} = 1 \quad \text{for all } n \geq 0
	\]
	This reflects the fact that there is exactly one way to choose nothing from a set of $n$ elements.
	
	\textbf{Sequence:} 1, 1, 1, 1, 1, 1, 1, ...
	
	\subsubsection*{2nd Diagonal: Natural Numbers}
	
	The second diagonal gives the natural numbers:
	\[
	\binom{n}{1} = n \quad \text{for all } n \geq 1
	\]
	There are $n$ ways to choose one item from $n$ options.
	
	\textbf{Sequence:} 1, 2, 3, 4, 5, 6, 7, ...
	
	\subsubsection*{3rd Diagonal: Triangular Numbers}
	
	The third diagonal reveals the triangular numbers:
	\[
	\binom{n}{2} = \frac{n(n-1)}{2} = T_{n-1} \quad \text{for all } n \geq 2
	\]
	These count the number of ways to choose two elements from $n$, and geometrically represent the number of dots needed to form an equilateral triangle.
	
	\textbf{Sequence:} 1, 3, 6, 10, 15, 21, 28, ...
	
	\subsection{General Pattern: \texorpdfstring{$k$}{k}-Dimensional Figurate Numbers}
	
	Each $(k+1)$-th diagonal in Pascal's Triangle contains the $k$-dimensional figurate numbers, also known as \textit{$k$-simplex numbers}:
	\[
	\binom{n}{k} = \frac{n!}{k!(n-k)!} \quad \text{for } n \geq k
	\]
	
	These numbers arise naturally in both combinatorics and geometry. They count the number of ways to form $k$-dimensional structures from $n$ objects:
	
	\begin{itemize}
		\item \textbf{0-simplex (point)}: All 1s ($\binom{n}{0}$)
		\item \textbf{1-simplex (line segment)}: Natural numbers ($\binom{n}{1}$)
		\item \textbf{2-simplex (triangle)}: Triangular numbers ($\binom{n}{2}$)
		\item \textbf{3-simplex (tetrahedron)}: Tetrahedral numbers ($\binom{n}{3}$)
		\item \textbf{4-simplex (pentatope)}: Pentatope numbers ($\binom{n}{4}$)
		\item \textbf{5-simplex:} 5D simplex numbers ($\binom{n}{5}$)
		\item \textbf{6-simplex:} 6D simplex numbers ($\binom{n}{6}$)
	\end{itemize}
	
	Each $k$-simplex number represents the number of ways to select $k$ objects from $n$, and geometrically corresponds to the number of vertices in a $k$-dimensional simplex formed from $n$ elements.

	\section{Main Theorem}
	
	In this section, we present our main result and illustrate it with a concrete example.
	
	\begin{theorem}\label{Thm:main-thm}
		Let \( m, n \in \mathbb{N} \) with \( 2 \leq m \leq n \), and let \( M = (m_{i,j})_{n \times n} \) be the matrix defined by
		\[
		m_{i,j} = 
		\begin{cases}
			\displaystyle \binom{m+n-i-1}{m-1}, & \text{if } i \geq j, \\[6pt]
			\displaystyle \binom{m+n-i-1}{m-1} - \binom{m-2 + j - i}{m-1}, & \text{if } i < j.
		\end{cases}
		\]
		Then \( \det(M) = 1 \).
	\end{theorem}
	
	\begin{example}[Pentatope Numbers in the Matrix]
		Let \( m = 5 \) and \( n = 7 \). The matrix \( M \) constructed from the binomial rule in Theorem~\ref{Thm:main-thm} takes the form of a \( 7 \times 7 \) matrix as shown below:
		\begin{figure}[H]
			\centering
			\begin{equation*}
				\begin{gathered}
					M_7 = 
					\begin{pmatrix}
						\binom{10}{4} & \binom{10}{4} - \binom{4}{4} & \binom{10}{4} - \binom{5}{4} & \binom{10}{4} - \binom{6}{4} & \binom{10}{4} - \binom{7}{4} & \binom{10}{4} - \binom{8}{4} & \binom{10}{4} - \binom{9}{4} \\[0.6em]
						\binom{9}{4} & \binom{9}{4} & \binom{9}{4} - \binom{4}{4} & \binom{9}{4} - \binom{5}{4} & \binom{9}{4} - \binom{6}{4} & \binom{9}{4} - \binom{7}{4} & \binom{9}{4} - \binom{8}{4} \\[0.6em]
						\binom{8}{4} & \binom{8}{4} & \binom{8}{4} & \binom{8}{4} - \binom{4}{4} & \binom{8}{4} - \binom{5}{4} & \binom{8}{4} - \binom{6}{4} & \binom{8}{4} - \binom{7}{4} \\[0.6em]
						\binom{7}{4} & \binom{7}{4} & \binom{7}{4} & \binom{7}{4} & \binom{7}{4} - \binom{4}{4} & \binom{7}{4} - \binom{5}{4} & \binom{7}{4} - \binom{6}{4} \\[0.6em]
						\binom{6}{4} & \binom{6}{4} & \binom{6}{4} & \binom{6}{4} & \binom{6}{4} & \binom{6}{4} - \binom{4}{4} & \binom{6}{4} - \binom{5}{4} \\[0.6em]
						\binom{5}{4} & \binom{5}{4} & \binom{5}{4} & \binom{5}{4} & \binom{5}{4} & \binom{5}{4} & \binom{5}{4} - \binom{4}{4} \\[0.6em]
						\binom{4}{4} & \binom{4}{4} & \binom{4}{4} & \binom{4}{4} & \binom{4}{4} & \binom{4}{4} & \binom{4}{4}
					\end{pmatrix} \\[1.5em]
					=
					\begin{tikzpicture}[>=Stealth, baseline=(m.center)]
						\matrix (m) [matrix of math nodes,
						left delimiter=(, right delimiter=),
						nodes={minimum width=0.9cm, minimum height=0.7cm, anchor=center},
						column sep=0.25cm, row sep=0.15cm]
						{
							210 & 209 & 205 & 195 & 175 & 140 & 84 \\
							126 & 126 & 125 & 121 & 111 & 91 & 56 \\
							70 & 70 & 70 & 69 & 65 & 55 & 35 \\
							35 & 35 & 35 & 35 & 34 & 30 & 20 \\
							15 & 15 & 15 & 15 & 15 & 14 & 10 \\
							5 & 5 & 5 & 5 & 5 & 5 & 4 \\
							1 & 1 & 1 & 1 & 1 & 1 & 1 \\
						};
						\draw[->, thick, red] (m-7-7.center) -- (m-6-6.center) -- (m-5-5.center) -- (m-4-4.center) -- (m-3-3.center) -- (m-2-2.center) -- (m-1-1.center);
						\draw[->, thick, blue] (m-7-6.center) -- (m-6-5.center) -- (m-5-4.center) -- (m-4-3.center) -- (m-3-2.center) -- (m-2-1.center);
						\draw[->, thick, green!60!black] (m-7-5.center) -- (m-6-4.center) -- (m-5-3.center) -- (m-4-2.center) -- (m-3-1.center);
						\draw[->, thick, orange] (m-7-4.center) -- (m-6-3.center) -- (m-5-2.center) -- (m-4-1.center);
						\draw[->, thick, purple] (m-7-3.center) -- (m-6-2.center) -- (m-5-1.center);
						\draw[->, thick, cyan] (m-7-2.center) -- (m-6-1.center);
					\end{tikzpicture}
				\end{gathered}
			\end{equation*}
			\caption{Matrix \( M \) with symbolic binomial entries and their simplified numerical form. Colored arrows indicate the diagonals corresponding to Pentatope numbers.}
			\label{fig:matrix_m3}
		\end{figure}
	\end{example}

	\section{Proof of main theorem}
	Combinatorial interpretations of determinants can provide deeper insight into their evaluations; this is especially true when the entries of a matrix admit natural
	graph-theoretic descriptions~\cite{18,14,11}. In this section, we present a bijective proof of our main result, Theorem \ref{Thm:main-thm}.
	Before proceeding to the proof, let us recall the celebrated \emph{Gessel–Lindström–Viennot} lemma (see~\cite{15} for details). For completeness, we reproduce the lemma from~\cite{15} below. Let $\Gamma$ be a weighted, acyclic digraph. The vertex set and the edge set of $\Gamma$ are denoted by $V(\Gamma)$ and $E(\Gamma)$, respectively.  
	A \emph{path} in $\Gamma$ is a sequence of distinct vertices $v_1, v_2, \ldots, v_r$ such that $(v_i, v_{i+1}) \in E(\Gamma)$ for $i = 1, \ldots, r-1$, where the edge is directed from $v_i$ to $v_{i+1}$.  
	For simplicity, we denote such a path by $v_1 v_2 \cdots v_r$, and an edge $(v_i, v_{i+1})$ by $v_i v_{i+1}$.
	The \emph{weight} of a path $P$, denoted by $w(P)$, is the product of the weights of all edges in $P,$ and the \emph{length} of $P$, denoted by $\ell(P)$, is the number of edges in $P$. 
	Suppose that $U = \{u_1, u_2, \ldots, u_n\}$ and $V = \{v_1, v_2, \ldots, v_n\}$ are two (not necessarily disjoint) $n$-element subsets of $V(\Gamma)$.  
	To these, we associate the \emph{path matrix} $M = (m_{ij})_{n \times n}$, where
	\[
	m_{ij} = \sum_{P : u_i \rightarrow v_j} w(P),
	\]
	and $P : u_i \rightarrow v_j$ denotes a path from $u_i$ to $v_j$.
	
	A \emph{path system} from $U$ to $V$ is an ordered pair $(\mathcal{P}, \sigma)$, where $\sigma$ is a permutation of $\{1, 2, \ldots, n\}$ and $\mathcal{P} = \{P_i : u_i \rightarrow v_{\sigma(i)} \mid 1 \le i \le n\}$ is a collection of $n$ paths.  
	The sign of a path system $(\mathcal{P}, \sigma)$ is $\operatorname{sgn}(\sigma)$, and its \emph{weight} is
	\[
	w(\mathcal{P}, \sigma) = \prod_{i=1}^n w(P_i).
	\]
	A path system is called \emph{vertex-disjoint} if no two paths in it share a common vertex. Let $VD_{\Gamma}$ denote the family of all vertex-disjoint path systems in $\Gamma$.
	
	\begin{lemma}[Gessel–Lindström–Viennot lemma, \cite{15}]\label{lgv-lemma}
		Let $\Gamma = (V(\Gamma), E(\Gamma))$ be a weighted, acyclic digraph.  
		Suppose $U = \{u_1, u_2, \ldots, u_n\}$ and $V = \{v_1, v_2, \ldots, v_n\}$ are two (not necessarily disjoint) $n$-element subsets of $V(\Gamma)$, and let $M$ be the corresponding path matrix from $U$ to $V$. Then
		\[
		\det(M) = \sum_{(\mathcal{P}, \sigma) \in VD_{\Gamma}} \operatorname{sgn}(\sigma) \, w(\mathcal{P}, \sigma),
		\]
		where $VD_{\Gamma}$ denote the family of all vertex-disjoint path systems in $\Gamma$.
	\end{lemma}
	
	\noindent
	Note that the sum is $0$ if no path system exists from $U$ to $V$.

	\begin{figure}[H]
		\centering
		\begin{tikzpicture}[scale=1,
			>=stealth,
			vertex/.style={circle, fill=blue!70, inner sep=2.5pt, draw=blue!80, thick}
			]
			\node at (0,0.3) {$\mathbf{u_1}$};
			\node at (2,0.3) {$\mathbf{u_2}$};
			\node at (4,0.3) {$\mathbf{u_3}$};
			\node at (6,0.3) {$\mathbf{u_4}$};
			\node at (9,0.3) {$\mathbf{u_{n-1}}$};
			\node at (11,0.3) {$\mathbf{u_n}$};
			\node at (0,-9.3) {$\mathbf{v_1}$};
			\node at (2,-9.3) {$\mathbf{v_2}$};
			\node at (4,-9.3) {$\mathbf{v_3}$};
			\node at (6,-9.3) {$\mathbf{v_4}$};
			\node at (9,-9.3) {$\mathbf{v_{n-1}}$};
			\node at (11,-9.3) {$\mathbf{v_n}$};
			\node at (-.3,-7.3) {$\mathbf{w_1}$};
			\node at (1.7,-7.3) {$\mathbf{w_2}$};
			\node at (3.7,-7.3) {$\mathbf{w_3}$};
			\node at (5.7,-7.3) {$\mathbf{w_4}$};
			\node at (8.5,-7.3) {$\mathbf{w_{n-1}}$};
			\node at (10.7,-7.3) {$\mathbf{w_n}$};
			\node at (-.3,-8.3) {$\mathbf{e_1}$};
			\node at (1.7,-8.3) {$\mathbf{e_2}$};
			\node at (3.7,-8.3) {$\mathbf{e_3}$};
			\node at (5.7,-8.3) {$\mathbf{e_4}$};
			\node at (8.5,-8.3) {$\mathbf{e_{n-1}}$};
			\node at (10.7,-8.3) {$\mathbf{e_n}$};
			\foreach \x in {0,2,4,6, 9, 11} {
				\foreach \y in {-0.2,-2,-4,-7,-9} {
					\fill (\x,\y) circle (2pt);
				}
			}
			
			\foreach \x in {0,2,4,6, 9, 11} {
				\draw[->, thick, green!70] (\x,-0.2) -- (\x,-1.9);
				\draw[->, thick, green] (\x,-2) -- (\x,-3.8);
				\draw[->, thick, green] (\x,-7) -- (\x,-8.9);  blue!70
				
				\draw[dashed] (\x,-4) -- (\x,-7);
			}
			
			\foreach \y in {-0.2,-2,-4,-7} {
				\draw[->, thick, blue] (0,\y) -- (1.9,\y);
				\draw[->, thick, blue] (2,\y) -- (3.9,\y);
				\draw[->, thick, blue] (4,\y) -- (5.9,\y);

				\draw[dashed] (6,\y) -- (8.9,\y);

					\draw[->, thick, blue] (8.9,\y) -- (10.9,\y);
					
				}
				\draw[->, thick, red] (1.9,-9) -- (.1,-9);
				\draw[->, thick, red] (3.9,-9) -- (2.1,-9);
				\draw[->, thick, red] (5.9,-9) -- (4.1,-9);
				\draw[dashed] (6,-9) -- (9,-9);
				\draw[->, thick,red] (10.9,-9) -- (9.1,-9);
				
			\end{tikzpicture}
			\caption{An \( m \times n \) directed, weighted grid graph in which each edge has weight \( 1 \). For each \( i \in [n] \), the edge from \( w_i \) to \( v_i \) is denoted by \( e_i \).}
			\label{fig: general grid graph}	
		\end{figure}
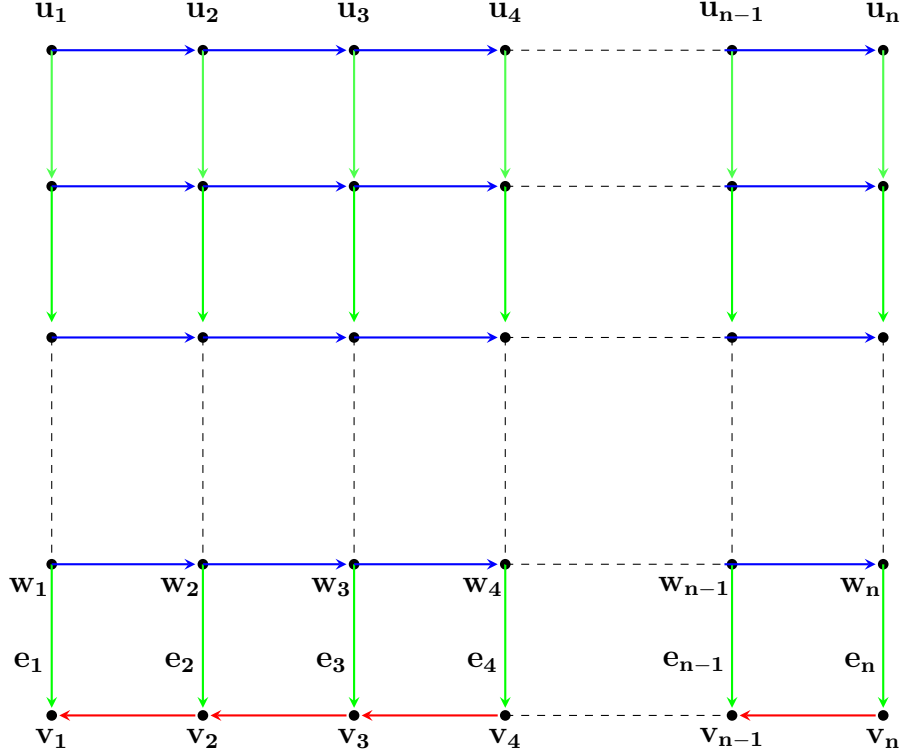
		
		In this section, we present the following theorem, which enumerates the number of lattice paths in the grid graph as depicted in Figure~\ref{fig: general grid graph}.
		
		\begin{theorem}\label{Thm:lattic-path-counting}
			Let \( \Gamma \) be the \( m \times n \) grid graph as defined in Figure~\ref{fig: general grid graph}, where \( m, n \in \mathbb{N} \) with \( 2 \leq m \leq n \), and assume that each edge has weight \( 1 \). Then the number of lattice paths from vertex \( u_i \) to vertex \( v_j \) is given by
			\[
			\ell_{i,j} =
			\begin{cases}
				\displaystyle\binom{m+n-i-1}{m-1}, & \text{if } i \geq j, \\[8pt]
				\displaystyle\binom{m+n-i-1}{m-1} - \binom{m - 2 + j - i}{m-1}, & \text{if } i < j.
			\end{cases}
			\]
		\end{theorem}
		
		\begin{proof}
			We give a bijective proof of this theorem. 
			
			\medskip
			\noindent\textbf{Case 1:} Let $i \geq j$. Then note that the number of lattice paths from $u_i$ to $v_i$ is equal to the number of lattice paths from $u_i$ to $v_j$ for any $j \leq i$. So, for this case, it is enough to compute the lattice paths from $u_i$ to $v_i$ in the grid graph described in Figure \ref{fig: general grid graph}. 
			
			Let $A$ be the set of all lattice paths from the vertex $u_i$ to the vertex $v_i$ in the $m \times n$ grid graph described in Figure~\ref{fig: general grid graph}. According to the direction of edges, one can easily see that one travels first from $u_i$ to $w_t$ only in the right and south directions, then by using south and left (red colored edges) directions, reaches the point $v_i$. For each $k \in [n]$, we denote by $A_k \subseteq A$ the set of all paths from $u_i$ to $v_i$ via the edge $e_k$. Then clearly, $A = \bigcup_{k=1}^n A_k$. Let $B$ be the set of all $(m-1)$-subsets of the set $Z = \{i, i+1, \ldots, m+n-2\}.$ We will show a bijection $f$ from $A$ to the set $B$. 
			
			Suppose a lattice path $L \in A = \bigcup_{k=1}^n A_k$. Then $L \in A_t$ for some $t \in [n]$. That is, $L$ is a lattice path from $u_i$ to $v_i$ via the edge $e_t$. Therefore, it takes exactly $(t-i) + (m-2)$ moves to get from $u_i$ to $w_t$. Among those moves, $(t-i)$ of them have to be going right (or east), denoted by $E$, and $(m-2)$ of them have to be going down (or south), denoted by $S$. Note that there is a bijective mapping between the set of paths connecting the points $u_i$ and $w_t$ and the set of distinct arrangements of $(t-i)$ $E$'s and $(m-2)$ $S$'s. 
			
			Now, let the $(m-2)$ south moves of $L$ occur at positions $x_1, x_2, \ldots, x_{m-2}$, where $x_1 < x_2 < \cdots < x_{m-2},$ and $1\leq x_s\leq (t-s)+(m-2),$ for all $s\in [m-2].$ The total number of south and east moves in the path $L$ is denoted by $e(L)$.
			Define the image of the path L under the mapping $f$ as $f(L) = Y = \{ y_1, y_2, \ldots, y_{m-1} \}$, where
			\begin{itemize}
				\item $y_d = x_d + (i-1)$ for all $d \in [m-2]$, and 
				\item $y_{m-1} =e(L)+(i-1)$.
			\end{itemize}
			Notice that $Y \subseteq Z$, so $f(L) \in B$.
			
			First, we show that $f \colon A \to B$ is an injection. Let $L_1 \neq L_2$ be two lattice paths from $u_i$ to $v_i$. If $\ell(L_1) \neq \ell(L_2)$, then clearly the $(m-1)$-th element of $f(L_1)$ is distinct from that of $f(L_2)$. Let $\ell(L_1) = \ell(L_2)$. Then from Figure~\ref{fig: general grid graph}, it can be noticed that either at least one $E$ move or at least one $S$ move of $L_1$ is distinct from that of $L_2$. This proves that $f(L_1) \neq f(L_2)$.
			
			Now we prove that $f$ is a surjective function. Let $Y = \{y_1, y_2, \ldots, y_{m-1}\} \subseteq Z$ with $y_1 < y_2 < \cdots < y_{m-1}$. We construct a path $L$ from $u_i$ to $v_i$ having the following properties:
			\begin{itemize}
				\item the first $(m-2)$ south moves occur at positions $y_1 - (i-1), y_2 - (i-1), \ldots, y_{m-2} - (i-1),$ and
				\item $e(L)=y_{m-1}-(i-1).$
				
			\end{itemize}
			Note that $|Z|=m+n-i+1.$ Therefore, $|A|=\binom{m+n-i+1}{m-1}.$ 

			\noindent
			\textbf{Case 2:} Let $i < j$. From Figure~\ref{fig: general grid graph}, it is clear that the number of lattice paths (using south and east moves only) from $u_i$ to $v_j$ equals 
			\[
			|A \setminus \bigcup_{r=i}^{j-1} A_r| = |A| - \Big|\bigcup_{r=i}^{j-1} A_r\Big|.
			\]
			Proceeding as in Case~1, we can establish a bijection between the sets $\bigcup_{r=i}^{j-1} A_r$ and $C$, where $C$ is the set of all $(m-1)$-element subsets of the set 
			\[
			T = \{i, i+1, \ldots, m+j-3\},
			\]
			with $i < i+1 < \cdots < m+j-3$. Since $|T| = m-2+j-i$, we have
			\[
			\Big|\bigcup_{r=i}^{j-1} A_r\Big| = \binom{m-2+j-i}{m-1}.
			\]
			This completes the proof.
		\end{proof}

		\begin{example}
			\begin{figure}[H]
				\centering
				\begin{tikzpicture}[scale=1.5,
					>=stealth,
					vertex/.style={circle, fill=black!70, inner sep=2.5pt, draw=black!80, thick}
					]
					
					\def\rows{5}      
					\def\cols{7}      
					
					\foreach \x in {0,...,\the\numexpr\cols-1} {
						\foreach \y in {0,...,\the\numexpr\rows-1} {
							\node[vertex] (v\x\y) at (\x,-\y) {};
						}
					}
					
					\foreach \y in {0,...,\the\numexpr\rows-2} {
						\foreach \x in {0,...,\the\numexpr\cols-2} {
							\draw[->, thick, blue!70] (v\x\y) -- (v\the\numexpr\x+1\relax\y);
						}
					}
					
					\foreach \y in {\the\numexpr\rows-1} {
						\foreach \x in {0,...,\the\numexpr\cols-2} {
							\draw[<-, thick, red!70] (v\x\y) -- (v\the\numexpr\x+1\relax\y);
						}
					}
					
					\foreach \x in {0,...,\the\numexpr\cols-1} {
						\foreach \y in {0,...,\the\numexpr\rows-2} {
							\draw[->, thick, green!70] (v\x\y) -- (v\x\the\numexpr\y+1\relax);
						}
					}
					
					\foreach \x in {0,...,\the\numexpr\cols-1} {
						\pgfmathtruncatemacro{\colnum}{\x+1}
						\node at (\x, 0.3) {$\mathbf{u_{\colnum}}$};
					}
					
					\foreach \x in {-.3,...,\the\numexpr\cols-1} {
						\pgfmathtruncatemacro{\colnum}{\x+2}
						\node at (\x,-3.2) {$\mathbf{w_{\colnum}}$};
					}
					
					\foreach \x in {-.2,...,\the\numexpr\cols-1} {
						\pgfmathtruncatemacro{\colnum}{\x+2}
						\node at (\x,-3.5) {$\mathbf{e_{\colnum}}$};
					}
					
					\foreach \x in {0,...,\the\numexpr\cols-1} {
						\pgfmathtruncatemacro{\colnum}{\x+1}
						\node at (\x,-4.3) {$\mathbf{v_{\colnum}}$};
					}
					\draw[->, line width=2.5pt, blue!70] (0, 0) -- (1, 0);
					\draw[->, line width=2.5pt, green!70] (1, 0) -- (1, -1);
					\draw[->, line width=2.5pt, green!70] (1, -1) -- (1, -2);
					\draw[->, line width=2.5pt, blue!70] (1, -2) -- (2, -2);
					\draw[->, line width=2.5pt, green!70] (2, -2) -- (2, -3);
					\draw[->, line width=2.5pt, blue!70] (2, -3) -- (3, -3);
					\draw[->, line width=2.5pt, green!70] (3, -3) -- (3, -4);
					\draw[->, line width=2.5pt, red!70] (3, -4) -- (2, -4);
					\draw[->, line width=2.5pt, red!70] (2, -4) -- (1, -4);
					\draw[->, line width=2.5pt, red!70] (1, -4) -- (0, -4);
					\node at (1.2,-.5) {\textcolor{pink}{$2$}};
					\node at (1.2,-1.5) {\textcolor{pink}{$3$}};
					\node at (2.2,-2.5) {\textcolor{pink}{$5$}};
					\draw[->, line width=2.5pt, green!70] (4, 0) -- (4, -1);
					\draw[->, line width=2.5pt, green!70] (4, -1) -- (4, -2);
					\draw[->, line width=2.5pt, blue!70] (4, -2) -- (5, -2);
					\draw[->, line width=2.5pt, blue!70] (5, -2) -- (6, -2);
					\draw[->, line width=2.5pt, green!70] (6, -2) -- (6, -3);
					\draw[->, line width=2.5pt, green!70] (6, -3) -- (6, -4);
					\draw[->, line width=2.5pt, red!70] (6, -4) -- (5, -4);
					\draw[->, line width=2.5pt, red!70] (5, -4) -- (4, -4);
					\node at (4.2,-.5) {\textcolor{pink}{$1$}};
					\node at (4.2,-1.5) {\textcolor{pink}{$2$}};
					\node at (6.2,-2.5) {\textcolor{pink}{$5$}};
					

				\end{tikzpicture}
				\caption{Each edge is assigned a weight of \( 1 \). The numbers highlighted in deep pink indicate the positions of the southward moves in the lattice paths.}
				\label{fig:examp}	
			\end{figure}
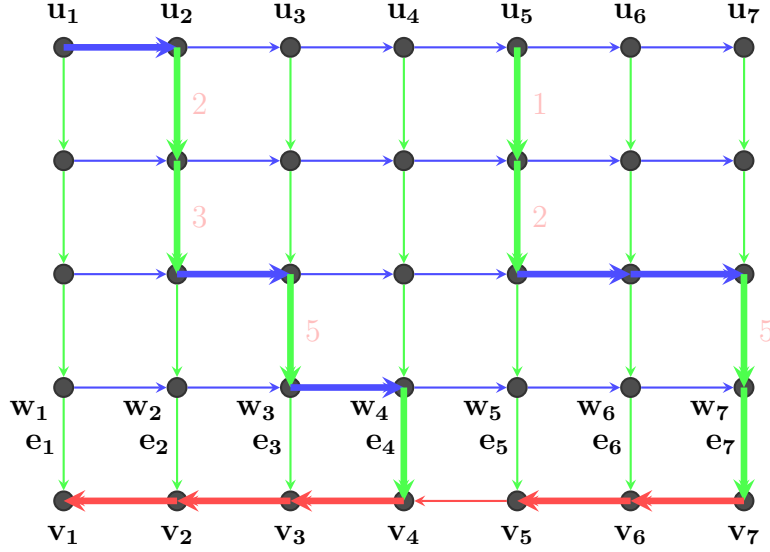
			
			Let \( L_1 \) and \( L_2 \) be two lattice paths from \( u_1 \) to \( v_1 \), and \( u_5 \) to \( v_5 \), respectively, described by the bold edges in Figure~\ref{fig:examp}. Clearly, \( L_1 \) reaches \( v_1 \) via the edge \( e_4 \), and all the south moves of \( L_1 \) occur at \( x_1 = 2, x_2 = 3, x_3 = 5 \) to reach \( w_4 \). 
			
			So, we compute:
			\[
			y_1 = x_1 + (1 - 1) = 2, \quad 
			y_2 = x_2 + (1 - 1) = 3, \quad 
			y_3 = x_3 + (1 - 1) = 5.
			\]
			
			Again, since \( e(L_1) = 7 \), we have:
			\[
			y_4 = e(L_1) + (1 - 1) = 7.
			\]
			
			Therefore, by the bijection described in the proof of Theorem~\ref{Thm:main-thm}, we get:
			\[
			f(L_1) = Y_1 = \{2, 3, 5, 7\} \subseteq \{1, 2, \ldots, 10\} = Z.
			\]
			
			Conversely, if we take the subset \( Y_1 = \{2, 3, 5, 7\} \subseteq Z \), then all the south moves of the corresponding lattice path \( L \) occur at positions \( 2, 3, 5 \) to reach some vertex \( w_t \). Since \( e(L) = 7 \), the path must pass through the edge \( e_4 \), which implies \( t = 4 \). Consequently, \( L = L_1 \).
			
			\bigskip
			
			Now, let \( i = 5 \). Then:
			\[
			Z = \{5, 6, 7, 8, 9, 10\}.
			\]
			
			Consider the subset \( Y_2 = \{5, 6, 9, 10\} \subseteq Z \). For this subset, we construct the unique path, say \( L_2 \), where the first \( (m - 2) = 3 \) south moves occur at the positions:
			\[
			x_1 = 5 - (5 - 1) = 1, \quad 
			x_2 = 6 - (5 - 1) = 2, \quad 
			x_3 = 9 - (5 - 1) = 5,
			\]
			and
			\[
			e(L_2) = 10 - (5 - 1) = 6.
			\]
			
			That is, \( L_2 \) should pass through the edge \( e_7 \), which matches the path described in Figure~\ref{fig:examp}.

		\end{example}
		\begin{lemma}\label{Lemma:VD-path-Characterization}
			Let \( \Gamma \) be the \( m \times n \) grid graph as depicted in Figure~\ref{fig: general grid graph}. Then a path system \( (\mathcal{P}, \sigma) \) is vertex-disjoint if and only if each of the \( n \) paths consists solely of southward moves.
		\end{lemma}
		
		\begin{proof}
			First, observe that the path system \( (\mathcal{P}, \sigma) \), where \( \sigma \) is the identity permutation and each of the \( n \) paths consists solely of southward moves, is vertex-disjoint.
			
			\medskip
			
			Conversely, let \( (\mathcal{P}, \sigma) \) be a path system from \( U = \{u_1, \ldots, u_n\} \) to \( V = \{v_1, \ldots, v_n\} \). We show that if at least one of the \( n \) paths contains an eastward move, then the path system cannot be vertex-disjoint.
			
			Let \( P_i \in \mathcal{P} \) be a path from \( u_i \) to \( v_j \) that includes an eastward move from \( (i, t) \) to \( (i+1, t) \), for some \( i \in [n-1] \) and \( t \in [m] \). Then the subsequent path \( P_{i+1} \), starting at \( u_{i+1} \), must avoid the vertex \( (i+1, t) \); otherwise, the paths \( P_i \) and \( P_{i+1} \) would share a vertex, violating vertex-disjointness.
			
			By continuing this argument inductively, we conclude that \( P_{n-1} \) must pass through \( (n, t) \). However, since \( \Gamma \) is the \( m \times n \) grid graph as described in Figure~\ref{fig: general grid graph}, the final path \( P_n \) must also pass through \( (n, t) \). Thus, the path system \( (\mathcal{P}, \sigma) \) cannot be vertex-disjoint.
			
			Now, suppose a path \( P_i \) includes a westward move (represented by a red edge in Figure~\ref{fig: general grid graph}). From the structure of the grid, it is clear that in order for such a move to occur, there must exist another path \( P_j \) that includes at least one eastward move. As shown above, the presence of an eastward move implies that the path system is not vertex-disjoint.
			
			\medskip
			
			This completes the proof.
		\end{proof}
		
		\begin{proof}[Proof of Theorem \ref{Thm:main-thm}]
			We now prove our main theorem using the celebrated Gessel-Lindström-Viennot lemma. Consider the grid graph depicted in Figure~\ref{fig: general grid graph}. Clearly, this is an acyclic, weighted, directed graph, where each edge has weight \( 1 \). 
			
			Therefore, by Theorem~\ref{Thm:lattic-path-counting}, the \( (i,j)^{\text{th}} \) entry of the associated path matrix \( M \) is given by
			\begin{align*}
				m_{i,j} =
				\begin{cases}
					\dbinom{m+n-i-1}{m-1}, & \text{if } i \geq j, \\
					\dbinom{m+n-i-1}{m-1} - \dbinom{m-2 + j - i}{m-1}, & \text{if } i < j.
				\end{cases}
			\end{align*}
			
			Now, by Lemma~\ref{Lemma:VD-path-Characterization}, the grid graph admits a unique vertex-disjoint path system, and the weight of this path system is clearly \( 1 \). Hence, by Lemma~\ref{lgv-lemma}, we have \( \det(M) = 1 \). This completes the proof.
			
		\end{proof}
		
		\section*{Acknowledgements} 
		The author would like to express sincere gratitude to the referee for their careful reading of the manuscript and for their valuable comments and suggestions, which have significantly improved the presentation and quality of this work.
		
		The author gratefully acknowledges the financial support provided through the NBHM Research Project (Reference No.~02011/29/2025NBHM(RP)/R\&DII/11951). The author also thanks the National Board for Higher Mathematics (NBHM), India, for supporting this research.
		
		The author is deeply grateful to Sajal Kumar Mukherjee for discussions related to this work.
		
		Finally, the author acknowledges the excellent research environment and facilities provided by Dhirubhai Ambani University, Gandhinagar, Gujarat, which greatly facilitated the completion of this research.
		
		\subsection*{Declaration of competing interest}
		The author have no relevant financial or non-financial interests to disclose. 
		
		\subsection*{Data availability}
		Data sharing is not applicable to this article as no datasets were generated or analyzed 
		during the present study.

		\bibliographystyle{amsplain}
		\bibliography{gen-inv-lcp20}
	\end{document}